\documentclass[preprint,12pt]{elsarticle}
\usepackage[utf8]{inputenc}
\usepackage[T1]{fontenc}
\usepackage[all,xdvi,2cell]{xy}\UseAllTwocells\SilentMatrices
\usepackage{amsmath}
\usepackage{amsfonts}
\usepackage{amsthm}
\usepackage{enumerate}
\usepackage{newunicodechar}
\newtheorem{theorem}{Theorem}[section]

\newtheorem{proposition}[theorem]{Proposition}

\theoremstyle{definition}
\newtheorem{definition}[theorem]{Definition}
\newtheorem{example}[theorem]{Example}
\newunicodechar{∖}{\backslash}
\newunicodechar{≤}{\leq}
\newunicodechar{α}{\alpha}
\newunicodechar{β}{\beta}

\newcommand{\id}{\mathrm{id}}
\newcommand{\card}{\mathrm{card}}
\newcommand{\C}{\mathcal{C}}
\newcommand{\D}{\mathcal{D}}
\newcommand{\newcategory}[1]{\expandafter\newcommand\csname #1\endcsname{\mathbf{#1}}}
\newcategory{Set}
\newcategory{BPos}
\newcategory{PsDPos}
\newcategory{Pos}
\newcategory{PsEA}
\newenvironment{acknowledgements}{\par\noindent\textbf{Acknowledgements:}}{}
\begin{document}
\journal{Fuzzy Sets and Systems}
\begin{frontmatter}
\title{Pseudo effect algebras are algebras over bounded posets}
\author{Gejza Jenča}
\address{
Department of Mathematics and Descriptive Geometry\\
Faculty of Civil Engineering,
Slovak University of Technology,
	Slovak Republic\\
              \texttt{gejza.jenca@stuba.sk}
}
\begin{abstract}
We prove that there is a monadic adjunction between the category of bounded posets
and the category of pseudo effect algebras. 
\end{abstract}
\end{frontmatter}

\section{Introduction}

In their 1994 paper~\cite{FouBen:EAaUQL}, D.J. Foulis and M.K. Bennett defined
effect algebras as (at that point in time) the most general version of
quantum logics. Their motivating example was the set of all Hilbert space
effects, a notion that plays an important role in quantum mechanics
\cite{Lud:FoQM,BusLahMit:TQToM}. An equivalent definition in terms of the difference
operation was independently given by F. Kôpka and F. Chovanec in~\cite{KopCho:DP}. Later
it turned out that both groups of authors rediscovered the definition given already in 1989
by R. Giuntini and H. Greuling in~\cite{GiuGre:TaFLfUP}.

By the very definition, the class of effect algebras
includes orthoalgebras~\cite{FouGreRut:FaSiO}, which include 
orthomodular posets and orthomodular lattices. It soon turned out~\cite{ChoKop:BDP}
that there is another interesting subclass of effect algebras, namely MV-algebras
defined by C.C. Chang in 1958~\cite{Cha:AAoMVL} to give the algebraic semantics
of the Łukasiewicz logic. Furthermore, K. Ravindran in his thesis~\cite{Rav:OaSToEA} proved that a
certain subclass of effect algebras (effect algebras with the Riesz decomposition property)
is equivalent with the class of partially ordered abelian groups with interpolation~\cite{Goo:POAGwI}.
This result generalizes the equivalence of MV-algebras and lattice ordered abelian groups described by
D. Mundici in~\cite{Mun:IoAFCSAiLSC}.

In \cite{kalmbach1977orthomodular}, Kalmbach proved the following theorem.
\begin{theorem}\label{thm:kalmbach}
Every bounded lattice $L$ can be embedded into an orthomodular lattice $K(L)$.
\end{theorem}
The proof of the theorem is constructive, $K(L)$ is known under the name {\em Kalmbach
extension} or {\em Kalmbach embedding}. In \cite{mayet1995classes}, Mayet and Navara proved
that Theorem \ref{thm:kalmbach} can be generalized: every bounded poset $P$ can be embedded
in an orthomodular poset $K(P)$. 
In fact, as proved by Harding in \cite{harding2004remarks}, this $K$ is then left adjoint to the forgetful
functor from orthomodular posets to bounded posets. This adjunction gives rise to a monad on the category
of bounded posets, which we call the {\em Kalmbach monad}.

For every monad $(T,\eta,\mu)$ on a category $\C$, there is a standard notion {\em Eilenberg-Moore} category $\C^T$
(sometimes called the {\em category of algebras} or the {\em category of modules} for $T$). The category
$\C^T$ comes equipped with a canonical adjunction between $\C$ and $\C^T$ and this adjunction 
gives rise to the original monad $T$ on $\C$. A functor equivalent to a right adjoint $\C^T\to\C$ is called
{\em monadic}.

In~\cite{jenca2015effect} the author proved that the Eilenberg-Moore category for the
Kalmbach monad is isomorphic to the category of effect algebras. In other words, the
forgetful functor from the category of effect algebras to the category of bounded posets
is monadic. 

In~\cite{dvurecenskij2001pseudoeffect}, Dvurečenskij and Vetterlein introduced pseudo effect
algebras, a non-commutative generalization of effect algebras. Many results
known for effect algebras were successfully generalized to the non-commutative case, let us
mention~\cite{dvurevcenskij2003central} generalizing some results from~\cite{Jen:ACBTTfEA,GreFouPul:TCoaEA}
and~\cite{foulis2010typepseudo} generalizing main results of~\cite{foulis2010type}.

In the present paper we continue this line of research. We generalize the main result of~\cite{jenca2015effect}
by proving that the forgetful functor $G$ from the category of pseudo effect algebras to the category of
bounded posets is monadic. Unlike in~\cite{jenca2015effect}, we shall not give an explicit description of the
left adjoint associated with $G$. Since we use Beck's monadicity theorem, the proof we present here
is shorter and simpler than the previous proof that covered only the commutative case.

\section{Preliminaries}
\subsection{Bounded posets}
A {\em bounded poset} is a structure $(P,\leq,0,1)$ such that
$\leq$ is a partial order on $P$ and $0,1\in P$ are the bottom and the top
elements of $(P,\leq)$, respectively.

Let $P_1,P_2$ be bounded posets. A map $\phi:P_1\to P_2$ is a 
{\em morphism of bounded posets} if and only if it satisfies the following
conditions.
\begin{itemize}
\item $\phi(1)=1$ and $\phi(0)=0$.
\item $\phi$ is isotone.
\end{itemize}

The category of bounded posets is denoted by $\BPos$.

\subsection{Pseudo effect algebras}
\begin{definition}\cite{dvurecenskij2001pseudoeffect}
A {\em pseudo effect algebra} is an algebra $A$ with a partial binary
operation $+$ and two constants $0,1$ such that, for all $a,b,c\in A$.
\begin{enumerate}[{(PE}1{)}]
\item If $a+(b+c)$ exists, then $(a+b)+c$ exists and
$a+(b+c)=(a+b)+c$.
\item There is exactly one $d$ and exactly one $e$ such that $a+d=e+a=1$. 
\item If $a+b$ exists, there are $d,e$ such that $d+a=b+e=a+b$.
\item If $a+1$ exists or $1+a$ exists, then $a=0$.
\end{enumerate}
\end{definition}

Every pseudo effect algebra can be equipped with a partial order given by the
rule $a\leq c$ if and only if there exists an element $b$ such that $a+b=c$. In this partial order,
$0$ is the smallest and $1$ is the greatest element, so every
pseudo effect algebra is a bounded poset.
A morphism $f\colon A\to B$ of pseudo effect algebras is a mapping
such that $f(0)=0$, $f(1)=1$ and whenever $a+b$ exists in $A$,
$f(a)+f(b)$ exists in $B$ and $f(a+b)=f(a)+f(b)$. The category
of pseudo effect algebras is denoted by $\PsEA$. Clearly,
every morphism of effect algebras is a morphism of the associated bounded
posets. 
A pseudo effect algebra is an {\em effect algebra}~\cite{FouBen:EAaUQL,GiuGre:TaFLfUP} if and only if it is
commutative.

Every closed interval $[0,u]$ in the positive cone of a partially ordered (not
necessarily abelian) group gives rise to an {\em interval pseudo effect algebra} \cite[Section 2]{dvurecenskij2001pseudoeffect}. It is well-known that the set of all automorphisms of
a poset equipped with composition and a partial order defined pointwise
forms a partially ordered group \cite[Example 1.3.19]{glass1999partially}.
Using these facts, it is easy to construct examples of non-commutative pseudo effect algebras:

\begin{example}
Let $E$ be the set of all strictly increasing functions (in other words, order-automorphisms of the poset 
$\mathbb R$) from $\mathbb R$ to $\mathbb R$ such that,
for all $x\in\mathbb R$, $x\leq f(x)\leq 2x$. Put $\mathbf 0:=\id_{\mathbb R}$ and $\mathbf 1:=(x\mapsto 2x)$.
For $f,g\in E$, define $f+g$ if and only if $f\circ g\in E$ and then put $f+g=f\circ g$. 
Then $(E,+,\mathbf 0,\mathbf 1)$ is a non-commutative pseudo effect algebra.
\end{example}

\subsection{Pseudo D-posets}

For our purposes, the axioms of pseudo effect algebras are not very handy. 
Rather that working with the partial operation $+$, it will be easier to start with 
a bounded poset and then to work with partial differences $/$ and $∖$, 
defined for every pair of comparable elements on a poset. 

\begin{definition}\cite{yun2004pseudo}
A {pseudo D-poset} is a bounded poset $(A,≤,0,1)$ with the smallest element $0$ and the greatest
element $1$, equipped with two partial operations $/$ and $∖$, such that $b/a$ and $b∖a$ are defined
if and only if $a≤b$ and, for all $a,b,c\in A$, the following conditions are satisfied.
\begin{enumerate}[{(PD}1{)}]
\item For any $a\in A$, $a/0=a∖0=a$.
\item If $a≤b≤c$, then $c/b≤c/a$ and $c∖b≤c∖a$, and we have
$(c/a)∖(c/b)=b/a$ and $(c∖a)/(c∖b)=b∖a$.
\end{enumerate}
\end{definition}

A pseudo D-poset is a {\em D-poset}~\cite{KopCho:DP} iff the partial operations $/$ and $∖$ coincide.

A morphism of pseudo D-posets is a morphism of bounded posets $f\colon A\to B$
such that, for all $a,b\in A$, $f(a/b)=f(a)/f(b)$ and $f(a∖b)=f(a)∖f(b)$. The category
of pseudo D-posets is denoted by $\PsDPos$. Clearly, there is a forgetful functor 
$G\colon\PsDPos\to\BPos$.

Let $A$ be a pseudo D-poset. A subset $B\subseteq A$ is a {\em subalgebra of $A$} if and only if
$0,1\in A$ and for all $a,b\in B$ such that $a\leq b$, the elements $a/b$ and $a∖b$ belong to $B$.

Generalizing the well-known fact that every effect algebra is a D-poset
and vice versa, it was proved by Yun, Yongmin, and Maoyin in~\cite{yun2004pseudo} that every pseudo effect algebra
is equivalent to a pseudo D-poset. Explicitly, if $(A,∖,/)$ is a pseudo D-poset, then
we may define a partial binary operation $+$ on $A$ so that $a+b$ is defined and equals $c$ if and only if
$b\leq c$ and $c/b=a$ if and only if $a\leq c$ and $c∖a=b$. Then $(A,+,0,1)$ is an effect algebra. On the other
hand, whenever $(A,+,0,1)$ is a pseudo effect algebra then for every $a\leq c$ we may
define $c/a$ and $c∖a$ by the rule $a+(c/a)=(c∖a)+a=c$. Moreover, these two constructions
are mutually inverse. That means that the categories $\PsEA$ and $\PsDPos$ are isomorphic.

\begin{proposition}\label{prop:PsDPosiscomplete}
The category $\PsDPos$ is small-complete.
\end{proposition}
\begin{proof}
It is easy to check that a product of every family of pseudo D-posets can be constructed as a product of
underlying bounded posets, the partial operation $∖$ and $/$ are then defined pointwise. For a parallel pair of
morphisms $f,g\colon A\to B$ in $\PsDPos$, their equalizer is the inclusion of a subalgebra $E=\{x\in A:f(x)=g(x)\}$
into $A$. Since $\PsDPos$ has all products and all equalizers, it has all small limits.
\end{proof}

\subsection{General adjoint functor theorem}

Adjoint functor theorems give conditions under which a continuous functor $G$
has a left adjoint $F$.  This allows us to avoid construction of the functor
$F$, which is sometimes a difficult endeavor.

\begin{theorem}\cite{freyd1964abelian}\cite[Theorem~V.6.2]{mac1998categories}
\label{thm:gaft}
Given a locally small, small-complete category $\D$, a functor
$G\colon\D\to\C$ is a right adjoint if and only if $G$ preserves
small limits and satisfies the following 

\underline{Solution Set Condition:} for each object $X$ of $\C$ there is a set $I$
and an $I$-indexed family of arrows $h_i\colon X\to G(A_i)$ such that
every arrow $h\colon X\to G(A)$ can be written as a composite
$h=G(j)\circ h_i$ for some $j\colon A_j\to A$. 
\end{theorem}

\subsection{Beck's monadicity theorem}

A functor $G\colon\D\to\C$ is {\em monadic} if and only if it is equivalent
to the forgetful functor from the category of algebras $\C^T$ to $\C$ for a monad $T$ on $\C$.
A colimit (or a limit) in a category $\C$ is {\em absolute} if and only if it is preserved by every
functor with domain $\C$.

\begin{theorem}\cite{beck1967triples}\cite[Theorem~VI.7.1]{mac1998categories}
\label{thm:beckmonadicity}
A functor $G\colon\D\to\C$ is monadic if and only if $G$ is a right adjoint
and $G$ creates coequalizers for those parallel pairs $f,g\colon A\to B$
in $\D$, for which
$$
\xymatrix{
G(A)
	\ar@<.5ex>[r]^{G(f)}
	\ar@<-.5ex>[r]_{G(g)}
&
G(B)
}
$$
has an absolute coequalizer in $\C$.
\end{theorem}
Beck's monadicity theorem is a device that allows us to prove that a functor is monadic without
having to explicitly describe the monad $T$ on $\C$ arising from the adjunction, 
describe its category of algebras $\C^T$ and to prove that $\C^T$ is equivalent to $\D$.

\section{The result}

\begin{theorem}
The forgetful functor $G\colon\PsDPos\to\BPos$ is monadic.
\end{theorem}
\begin{proof}

Let us apply Theorem~\ref{thm:gaft} to prove $G$ is a right adjoint. By Proposition~\ref{prop:PsDPosiscomplete},
$\PsDPos$ is small-complete.
It is easy to check that $G$ preserves all small limits, since pseudo D-posets are algebraic structures;
the partiality of the operations is not a problem here. Let us check the Solution
Set Condition. Let $P$ be a bounded poset. Let $\mathcal W_P$ be a set of bounded posets
such that for every bounded poset $W'$ with 
$\card(P)\leq\card(W')\leq \max(\card(P),\aleph_0)$, there is a $W\in\mathcal W_P$ such that
$W$ is isomorphic to $W'$. Consider
the family $\mathcal H_P=\{h_i\}_{i\in I}$ of all $\BPos$-morphisms $h_i\colon P\to G(A_i)$, where $A_i$ is a 
pseudo D-poset and $G(A_i)\in\mathcal Q_P$. For every $\BPos$-morphism
$h\colon P\to G(A)$, the cardinality of the subalgebra $B$ of $A$ that is generated by the range of $h$
is bounded below by $\card(P)$ and above by $\max(\card(P),\aleph_0)$. Write $j\colon B\to A$
for the embedding of the subalgebra $B$ into $A$. Clearly, $h=G(j)\circ h_i$ for some
$h_i\in M$. Since $G$ preserves small limits and the Solution Set Condition is satisfied,
$G$ is a right adjoint.

We have proved that $G$ is a right adjoint, so we may apply Theorem~\ref{thm:beckmonadicity}.
Let $A,B$ be pseudo D-posets, let $f,g\colon A\to B$ be morphisms of
pseudo D-posets. Suppose that
\begin{equation}
\label{diag:absolute}
\xymatrix{
G(A)
	\ar@<.5ex>[r]^{G(f)}
	\ar@<-.5ex>[r]_{G(g)}
&
G(B)
	\ar[r]^q
&
Q
}
\end{equation}
is an absolute coequalizer. Assuming this, we need to prove that
there is a unique morphism of pseudo D-posets $q'\colon B\to Q'$ such that
\begin{equation}
\xymatrix{
A
	\ar@<.5ex>[r]^{f}
	\ar@<-.5ex>[r]_{g}
&
B
	\ar[r]^{q'}
&
Q'
}
\end{equation}
is a coequalizer in $\PsDPos$ and $Q=G(Q')$, $q=G(q')$. Let us 
prove that such $q'$ exists. We use the fact that~\eqref{diag:absolute}
is an absolute coequalizer to equip the bounded poset $Q$ with a structure of a pseudo D-poset. 
Then we prove that $q$ comes from 
a morphism of pseudo D-posets. Finally, we prove that this morphism of
pseudo D-posets is a coequalizer of $f,g$ in $\PsDPos$.

For every poset $P$, let us write $I(P)$ for the set of comparable pairs $\{(a,b)\in P\times P\colon a≤b\}$ and partially
order $I(P)$ by the rule
$(a,b)≤(c,d)$ if and only if $c≤a≤b≤d$. Note that the elements of
$I(P)$ can be identified with closed intervals of $P$, ordered by inclusion.
We shall write $[a≤b]$ for the element $(a,b)$ of $I(P)$.
The construction $P\mapsto I(P)$ can be made into a functor $\Pos\to\Pos$ by the rule
$I(f)([a≤b])=[f(a)≤f(b)]$.

In what follows, we write $U\colon\BPos\to\Pos$ for the obviously defined forgetful functor
from bounded posets to posets.
Note that, for every pseudo D-poset $X$,
the partial operation $/$ can be described as an isotone map $/_X\colon IUG(X)\to UG(X)$:
for every $[a≤b]\in IUG(X)$, $/_X([a≤b])=a/b$.

Moreover, for every morphism of pseudo D-posets $h\colon X\to Y$, the squares
\begin{equation}
\xymatrix@C=3pc{
IUG(X)
	\ar[r]^{IUG(h)}
	\ar[d]_{/_X}
&
IUG(Y)
	\ar[d]^{/_Y}
\\
UG(X)
	\ar[r]_{UG(h)}
&
UG(Y)
}
\qquad
\xymatrix@C=3pc{
IUG(X)
	\ar[r]^{IUG(h)}
	\ar[d]_{∖_X}
&
IUG(Y)
	\ar[d]^{∖_Y}
\\
UG(X)
	\ar[r]_{UG(h)}
&
UG(Y)
}
\end{equation}
commute. Indeed, this is just a reformulation of the assumption that $h$ is a morphism of pseudo D-posets. 
Therefore, the 
families of $\Pos$-morphisms 
\[
(/_X)_{X\in ob(\PsDPos)}\qquad
(∖_X)_{X\in ob(\PsDPos)}
\]
form a pair of natural transformations
from functor $IUG\colon\PsDPos\to\Pos$ to functor $UG\colon\PsDPos\to\Pos$.
Thus, both $/$ and $∖$ are morphisms in the category of functors $[\PsDPos,\Pos]$ with
source $IUG$ and target $UG$.

Let us focus on the partial operation $/$ (or, as explained in the previous paragraph, a natural
transformation $/$).
Consider the diagram
\begin{equation}
\label{diag:defineslash}
\xymatrix@C=3pc{
IUG(A)
	\ar@<.5ex>[r]^{IUG(f)}
	\ar@<-.5ex>[r]_{IUG(g)}
	\ar[d]_{/_A}
&
IUG(B)
	\ar[r]^{IU(q)}
	\ar[d]^{/_B}
&
IU(Q)
	\ar@{.>}[d]^{/}
\\
UG(A)
	\ar@<.5ex>[r]^{UG(f)}
	\ar@<-.5ex>[r]_{UG(g)}
&
UG(B)
	\ar[r]^{U(q)}
&
U(Q)
}
\end{equation}
Since $f,g$ in diagram~\eqref{diag:absolute} are morphisms of pseudo D-posets,
the naturality of $/$ implies that both the $f$ and $g$ left-hand squares in~\eqref{diag:defineslash}
commute. Since the coequalizer~\eqref{diag:absolute} is absolute,
both the top and the bottom rows in~\eqref{diag:defineslash} are coequalizers in $\Pos$. From the commutativity of
both $f$ and $g$ 
left-hand squares and from the fact that the bottom row is a coequalizer,
it follows that the morphism $U(q)\circ /_B\colon IUG(B)\to U(Q)$ coequalizes
the top pair of parallel arrows. Indeed,
\begin{equation*}
\begin{split}
U(q)\circ/_B\circ IUG(f)&=U(q)\circ UG(f)\circ /_A\\&=U(q)\circ UG(g)\circ /_A\\&= U(q)\circ/_B\circ IUG(g).
\end{split}
\end{equation*}

Since the top row in~\eqref{diag:defineslash} is a coequalizer, there is
a unique arrow $/\colon IU(Q)\to U(Q)$ making the right square commute.
Note that, actually, we equipped the bounded 
poset $Q$ with a partial binary operation $/$, that
is defined for all comparable pairs of elements of $Q$. In an analogous way,
we may define a partial operation $∖$ on $Q$.

Let us prove that these partial operations on $Q$ satisfy the axioms of a
pseudo D-poset. 
For every bounded poset $P$, let $[0≤\_]_P$ denote the mapping from $U(P)$ to
$IU(P)$ given by the rule $x\mapsto[0≤x]$. It is easy to see that this $ob(\BPos)$-indexed
family of arrows forms a natural transformation
from $U$ to $IU$. Moreover, the $/$ half of (PD1) is equivalent with the fact that, for every pseudo D-poset $X$,
the diagram
\begin{equation}
\label{diag:PD1}
\xymatrix@C=3.5pc{
UG(X)
	\ar[r]^{[0\,≤\,\_]_{G(X)}}
	\ar[d]_{\id}
&
IUG(X)
	\ar[ld]^{/_X}
\\
UG(X)
}
\end{equation}
commutes, so we see that $/\circ([0≤\_]*G)=\id_{UG}$ in the category of functors $[\PsDPos,\Pos]$. Similarly,
$∖\circ([0≤\_]*G)=\id_{UG}$.

To prove that the partial operation $/$ on $Q$ satisfies (PD1), consider the diagram
\begin{equation}
\label{diag:PD1transfer1}
\xymatrix@C=3pc{
UG(A)
	\ar@<.5ex>[r]^{UG(f)}
	\ar@<-.5ex>[r]_{UG(g)}
	\ar[d]_{[0\,≤\,\_]_{G(A)}}
&
UG(B)
	\ar[r]^{U(q)}
	\ar[d]^{[0\,≤\,\_]_{G(B)}}
&
U(Q)
	\ar[d]^{[0\,≤\,\_]_Q}
\\
IUG(A)
	\ar@<.5ex>[r]^{IUG(f)}
	\ar@<-.5ex>[r]_{IUG(g)}
	\ar[d]_{/_A}
&
IUG(B)
	\ar[r]^{IU(q)}
	\ar[d]^{/_B}
&
IU(Q)
	\ar[d]^{/}
\\
UG(A)
	\ar@<.5ex>[r]^{UG(f)}
	\ar@<-.5ex>[r]_{UG(g)}
&
UG(B)
	\ar[r]^{U(q)}
&
U(Q)
}
\end{equation}

By the commutativity of~\eqref{diag:PD1}, we see that the left and middle verticals
in~\eqref{diag:PD1transfer1} compose
to $\id_{UG(A)}$ and $\id_{UG(B)}$, respectively. Merging the vertical 
squares gives us the diagram
\begin{equation}
\label{diag:PD1transfer2}
\xymatrix@C=3pc{
UG(A)
	\ar@<.5ex>[r]^{UG(f)}
	\ar@<-.5ex>[r]_{UG(g)}
	\ar[d]_{\id_{UG(A)}}
&
UG(B)
	\ar[r]^{U(q)}
	\ar[d]^{\id_{UG(B)}}
&
U(Q)
	\ar[d]^{/\circ[0\,≤\,\_]_Q}
\\
UG(A)
	\ar@<.5ex>[r]^{UG(f)}
	\ar@<-.5ex>[r]_{UG(g)}
&
UG(B)
	\ar[r]^{U(q)}
&
U(Q)
}
\end{equation}

Note that if we replace the rightmost vertical arrow in~\eqref{diag:PD1transfer2}
by $\id_{U(Q)}$, the diagram still commutes. However, by an analogous
argument we have used to define $/$ on $Q$, the rightmost vertical arrow in
\eqref{diag:PD1transfer2} is unique. Therefore, $/\circ[0≤\_]=\id_{U(Q)}$, 
that means, for all $x\in Q$, $x/0=x$. The other half of (PD1) follows similarly.

Let us prove (PD2). For every poset $P$, let $J(P)$ be a poset consisting of
all comparable triples $[x≤y≤z]$ of $P$, partially ordered by the
rule
\begin{gather*}
[x_1≤y_1≤z_1]≤[x_2≤y_2≤z_2]\\
\Updownarrow\\
x_2≤x_1\text{ and }y_1≤y_2\text{ and }z_1=z_2.
\end{gather*}
For every morphism of posets $f\colon P\to Q$, let us define
$J(f)\colon J(P)\to J(Q)$ pointwise: 
\[
J(f)([x≤y≤z])=[f(x)≤f(y)≤f(z)].
\]
Obviously, $J\colon\Pos\to\Pos$ is
a functor.

For every bounded poset $P$, let $α_P\colon JU(P)\to IIU(P)$ be a map
given by the rule $α_P([x≤y≤z])=[[y≤z]≤[x≤z]]$ and let
$β_P\colon JU(P)\to IU(P)$ be a map given by the rule $β_P([x≤y≤z])=[x≤y]$.
Note that both maps $α_P$ and $β_P$ are isotone. Moreover, the families $α_{\_}$ and
$β_{\_}$ indexed by the objects of $\BPos$ form natural transformations $α\colon JU\to IIU$ and $β\colon JU\to IU$.

We may now express one half of the (PD2) axiom by a commutative diagram; for every pseudo D-poset $X$
the diagram
\begin{equation}
\label{diag:PD2}
\xymatrix@C=2.5pc{
JUG(X)
	\ar[rr]^-{β_{G(X)}}
	\ar[d]_-{α_{G(X)}}
&
&
IUG(X)
	\ar[d]^-{/_X}
\\
IIUG(X)
	\ar[r]_-{I(/_X)}
&
IUG(X)
	\ar[r]_-{∖_X}
&
UG(X)
}
\end{equation}
commutes. Indeed, chasing an element $[x≤y≤z]\in JUG(X)$ around~\eqref{diag:PD2} gives us
\begin{equation}
\xymatrix@C=2.5pc{
[x≤y≤z]
	\ar@{|->}[rr]^-{β_{G(X)}}
	\ar@{|->}[d]_-{α_{G(X)}}
&
&
[x≤y]
	\ar@{|->}[d]^-{/_X}
\\
[[y≤z]≤[x≤z]]
	\ar@{|->}[r]_-{I(/_X)}
&
[(z/y)≤(z/x)]
	\ar@{|->}[r]_-{∖_X}
&
(z/y)∖(z/x)=(y/x)
}
\end{equation}

This shows that, in the category of functors $[\PsDPos,\Pos]$, $/\circ(β*G)=∖\circ(I*/)\circα$.
We may now give a similar argument as we did to prove (PD1) that the partial operations $/,∖$ on $Q$ satisfy
the (PD2) axiom. 

We have proved that the partial operations $/$ and $∖$ we defined on $Q$ satisfy the
axioms of a pseudo D-poset.
In other words, there is a pseudo D-poset $Q'$ such that $Q=G(Q')$.
Moreover, the
morphism $q\colon U(B)\to Q=U(Q')$ of bounded posets preserves $/$,
since the right-hand square of~\eqref{diag:defineslash} commutes. Since $q$ preserves
$∖$ as well, we see that $q=U(q')$ for a morphism of pseudo D-posets $q'\colon B\to Q'$.
With this fact in mind, we may now observe that the diagram~\eqref{diag:defineslash} and
its $∖$-twin mean that $q'\circ f=q'\circ g$ in $\PsDPos$ and since the pseudo D-poset
structure on $Q$ arising from those diagrams is unique, we see that $Q'$ is unique.
Uniqueness of $q'$ follows from the fact that $G$ is a faithful functor.

Let us prove that $q'$ is a coequalizer of the pair $f,g$ in $\PsDPos$.
Let $h\colon B\to C$ be a morphism of pseudo D-posets such that $h\circ f=h\circ g$. Since
the diagram~\eqref{diag:absolute} is a coequalizer, there is a unique morphism of bounded posets
$e\colon G(Q)\to G(C)$ such that $e\circ G(q')=e\circ q=G(h)$. It remains to prove that this $e$ preserves
the partial operations $/$ and $∖$ on $Q$. Consider the diagram
\begin{equation}
\label{diag:final}
\xymatrix{
IUG(B)
	\ar[r]^{IU(q)}
	\ar[d]_{/_B}
&
IU(Q)
	\ar[d]^{/}
	\ar[r]^{IU(e)}
&
IUG(C)
	\ar[d]^{/_C}
\\
UG(B)
	\ar[r]_{U(q)}
&
U(Q)
	\ar[r]_{U(e)}
&
UG(C)
}
\end{equation}

We need to prove that the right-hand square of~\eqref{diag:final} commutes. By the commutativity
of the diagram~\eqref{diag:defineslash}, we already know that
the left hand square of~\eqref{diag:final} commutes. 
As $G(h)=e\circ G(q')$ and $h$ is a morphism of pseudo
D-posets, the outer square of~\eqref{diag:final} commutes. Therefore,
$$
U(e)\circ U(q)\circ /_B=/_C\circ IU(e)\circ IU(q)=U(e)\circ/\circ IU(q)
$$
Since the top row in~\eqref{diag:defineslash} is a coequalizer, $IU(q)$ 
is a coequalizer and thus an epimorphism. This implies that $/_C\circ IU(e)=U(e)\circ /$
and we see that the right-hand square of~\eqref{diag:final} commutes.
\end{proof}
\begin{acknowledgements}
The author is indebted to both referees for valuable comments on earlier drafts
that helped to improve the paper.

This research is supported by grants VEGA 2/0069/16 and 1/0006/19,
Slovakia and by the Slovak Research and Development Agency under the contracts
APVV-18-0052 and APVV-16-0073.
\end{acknowledgements}

\end{document}